\documentclass[12pt,a4paper,twoside]{amsart}
\usepackage{geometry} 
\usepackage{amsmath}
\usepackage{amsthm}
\usepackage{amsfonts,amssymb}
\usepackage{fancyhdr} 
\usepackage{amsfonts,amssymb}  
\usepackage{mathrsfs}  
\usepackage{graphicx}
\usepackage[all]{xy}
\usepackage{times}

\usepackage{float} 
\usepackage{amsthm}
\renewcommand\proofname{Proof}

\newtheorem{lemma}{Lemma}
\newtheorem{thm}{Theorem}

\newtheorem{prop}{Proposition}
\newtheorem*{claim}{Claim}

\newtheorem*{nota}{Notations}

\newcommand{\IC}{{\mathbb C}}

\newcommand{\IK}{{\mathbb K}}

\newcommand{\IP}{{\mathbb P}} 
\newcommand{\IQ}{{\mathbb Q}} 
\newcommand{\IR}{{\mathbb R}}

\newcommand{\IZ}{{\mathbb Z}}

\newcommand{\CA}{{\mathcal A}}
\newcommand{\CB}{{\mathcal B}}

\newcommand{\CH}{{\mathcal H}}

\newcommand{\CO}{{\mathcal O}} 
\newcommand{\CP}{{\mathcal P}}

\newcommand{\CX}{{\mathcal X}}

\newcommand{\seq}{\subseteq}
\newcommand{\lam}{\lambda}

\newcommand{\la}{\langle}
\newcommand{\ra}{\rangle}
\newcommand{\ti}{\times}
\newcommand{\D}{\Delta}
\newcommand{\de}{\delta}
\newcommand{\ga}{\gamma}

\begin{document}
\title{Dimension four simply connected Voisin manifolds}
\author{Linsheng Wang}
\curraddr{}
\email{linsheng\_wang@outlook.com}
\thanks{}
\date{}
\dedicatory{}
\begin{abstract}
Voisin constructed a series of examples concerning simply connected compact Kähler manifolds of even dimensions, which do not have the rational homotopy type of a complex projective manifold starting from dimension six. In this note, we prove that Voisin's examples of dimension four also does not have the rational homotopy type of a complex projective manifold. Oguiso constructed simply connected compact Kähler manifolds starting from dimension four, which can not deform to a complex projective manifold under a small deformation. We also prove that Oguiso's examples do not have the rational homotopy type of a complex projective manifold. 
\end{abstract}
\maketitle

\tableofcontents

\section{Introduction}

The Kodaira problem is that whether a compact Kähler manifold admits algebraic approximation. Here an \emph{algebraic approximation} of a compact Kähler manifold $X$ is a deformation $\CX\to \Delta$ of $X$ such that the subset of $\Delta$ parametrizing projective manifolds is dense. Kodaira \cite{Kodaira60} proved that any compact Kähler surface admits algebraic approximation. Recently Lin \cite{Lin17} proved an analogous result of Kähler threefolds. 

But starting from dimension $4$, Voisin \cite{Voisin04} constructed compact Kähler manifolds which do not have the rational homotopy type of, in particular, aren't deformation equivalent to a complex projective manifold. Here the \emph{rational homotopy type} of a manifold is the rational cohomology ring structure of it. Voisin \cite{Voisin04} also constructed simply connected examples of even dimensions starting from dimension $6$. 

In this note, we prove that the simply connected manifold $X_2$ of dimension $4$ constructed in \cite{Voisin04} also does not have the rational homotopy type of a complex projetive manifold. We also prove that $X_2\times M$ does not have the rational homotopy type of a complex projective manifold for any compact Kähler manifold $M$ with second Betti number $b_2(M)=1$. 

Oguiso \cite{Oguiso08} constructed simply connected examples using special automorphisms of K3 surfaces, which are rigid, hence admit no algebraic approximation. We will prove that Oguiso's examples do not have the rational homotopy type of a complex projective manifold. 

{}
\begin{thm}
In any dimension $\ge 4$, there exists a simply connected compact Kähler manifold which do not have the rational homotopy type of a complex projective manifold. 
\end{thm}

\begin{nota} \rm
For a free $\IZ$-module $V$, we denote $V_\IK=V\otimes_\IZ \IK$ where $\IK=\IQ, \IR$ or $\IC$. Let $X$ be a compact Kähler manifold of dimension $d$. We simply write $H^{p,p}(X,\IK)=H^{2p}(X, \IK) \cap H^{p,p}(X)$ and $(H^{p,q}\oplus H^{q,p})(X,\IK)=(H^{p,q}(X)\oplus H^{q,p}(X))\cap H^{p+q}(X,\IK)$ for $\IK=\IQ$ or $\IR$. With these notations we have $NS(X)=H^{1,1}(X,\IZ)$. For any Kähler class $c\in H^2(X,\IR)$, we use $q_c$ to denote the quadratic form $\alpha\mapsto c^{d-2}\alpha^2$ on $H^2(X,\IR)$. 
\end{nota}

\noindent {\it Acknowledgments}. I would like to express my gratitude to my advisor, Gang Tian, for his support and guidance. I thank Claire Voisin for telling me the Oguiso's example. I also thank Qizheng Yin for helpful discussions, Fan Ye for his careful reading.

\section{Preliminaries}

In this section, we state some basic results in \cite{Voisin04} and prove lemmas that we need in the next section. 

\subsection{Complex tori with endomorphisms}
Let $\Gamma = \mathbb{Z}^{2n}$ be a lattice, and let $\phi: \Gamma \to \Gamma$ be a $\mathbb{Z} $-linear map with characteristic polynomial $f(\lambda)=\prod_{i=1}^n(\lambda - \lambda_i)(\lambda - \overline{\lambda_i})$, where $\lambda_1,\dots,\lambda_n,\,\,$ $\overline{\lambda_1},$$\dots,$ $\overline{\lambda_n}$ are distinct non-real complex numbers. Then $\Gamma_\mathbb{C}=\Gamma \otimes_\mathbb{Z} \mathbb{C}$ has a decomposition $\Gamma_\mathbb{C} = \Gamma' \oplus \overline{\Gamma'}$, where $\Gamma'$ is the eigenspace of eigenvalues $\{\lambda_1,\dots,\lambda_n\}$. Now we have a complex torus $T=\Gamma_\mathbb{C}/(\Gamma'\oplus \Gamma)$ of complex dimension $n$ with an endomorphism $\phi_T$ induced by $\phi$.  



Let $(V,\psi) = (\Gamma_\IQ^*, \,^t\phi)$. Then we have: 

\begin{lemma} \label{no Hodge class lemma}
Suppose the Galois group of $f$ is isomorphic to the $2n$-th symmetric group. Then any Hodge structure of weight $2$ with non-trivial $(2,0)$ part on $\wedge^2V$ preserving by $\wedge^2\psi$ contains no Hodge class. 
\end{lemma}

\begin{proof}
We first note that the set of eigenvalues of $\psi$ is $\{\lambda_i, \overline{\lambda_i}\}_{i=1}^n$, on which ${\rm Gal}(f)$ acts as the $2n$-th symmetric group. So  the set of eigenvalues of $\wedge^2\psi$ is $S=\{\lambda_i\lambda_j, \overline{\lambda_i\lambda_j}\}_{i<j} \cup \{\lambda_i\overline{\lambda_j}\}_{i,j=1}^n$, on which ${\rm Gal}(f)$ acts transitively. 

Suppose $\wedge^2V$ admits a Hodge structure preserving by $\wedge^2\psi$. Then $(\wedge^2V)^{p,q}$ is an invariant subspace of $(\wedge^2\psi)_\IC$. We denote the set of eigenvalues of $(\wedge^2\psi)_\IC$ on $(\wedge^2V)^{p,q}$ by $S^{p,q}$, so $S=S^{2,0}\cup S^{1,1}\cup S^{0,2} $. Since all the eigenvalues of $\wedge^2 \psi$ are distinct and of multiplicity $1$, the three subsets do not intersect with each other. 

Let $S'$ be the set of eigenvalues of $\wedge^2 \psi$ on $\wedge^2V \cap (\wedge^2V)^{1,1}$. We have $S' \seq S^{1,1}$, so $S' \cap S^{2,0}= \emptyset$. Since $\wedge^2V \cap (\wedge^2V)^{1,1}$ is a $\IQ$-vector space, the following lemma of Galois theory says that $S'$ is stable under the action of ${\rm Gal}(f)$. But by assumption $S^{2,0}\ne \emptyset$, and ${\rm Gal}(f)$ acts transitively on $S$. We must have $S'=\emptyset$. So $\wedge^2V \cap (\wedge^2V)^{1,1}= 0$. 
\end{proof}

We will always assume that ${\rm Gal}(f)$ is isomorphic to the $2n$-th symmetric group in the following. We have a canonical isomorphism $H_1(T,\IZ) = \Gamma$, so by Poincare duality we have $V=H^1(T, \IQ)$ and $\psi=\phi_T^*$. From Lemma \ref{no Hodge class lemma} we see that $NS(T)=0$, so the meromorphic function field of $T$ is $\IC$, and $T$ is not projective.

\begin{lemma}
Let $W$ be a $\IQ$-vector space, $A:W\to W$ a $\IQ$-linear map. We denote $f(\lam)={\rm det}(\lam I-A)$, and $K$ to be the splitting field of $f$ over $\IQ$. If $W_0 \seq W$ is an $A$-invariant subspace. Let $f_0(\lam) = {\rm det}(\lam I-A|_{W_0})$, $K_0 \seq K$ the splitting field of $f_0$ over $\IQ$. Then there exists a natural morphism of groups
$${\rm Gal}(K/\IQ) \to {\rm Gal}(K_0/\IQ). $$
\end{lemma}
This lemma says that the set of eigenvalues of $A|_{W_0}$ (the set of roots of $f_0$) is stable under the action of ${\rm Gal}(f)$. 

\begin{proof}
Since $\IQ$ is a perfect field, the two splitting fields $K$ and $K_0$ are both finite Galois extensions of $\IQ$. By the fundamental theorem of Galois theory, ${\rm Gal}(K/K_0) \seq {\rm Gal}(K/\IQ)$ is a normal subgroup. So for any $\sigma \in {\rm Gal}(K/\IQ),  \delta \in {\rm Gal}(K/K_0)$, there exist $\delta' \in {\rm Gal}(K/K_0)$ such that $\delta\sigma=\sigma\delta'$. Hence for any $a\in K_0, \delta \in {\rm Gal}(K/K_0)$, we have $\delta (\sigma(a))=\sigma(a)$. So $\sigma(a) \in {\rm Inv}({\rm Gal}(K/K_0))=K_0$. We see that the morphism ${\rm Gal}(K/\IQ) \to {\rm Gal}(K_0/\IQ), \sigma \mapsto \sigma|_{K_0}$ is well-defined. 
\end{proof}

\subsection{Deligne's lemma}
Let $\CA^*=\oplus_{k\ge0}\CA^k $ be the rational cohomology ring of a compact Kähler manifold. Suppose $Z \seq \CA^k_\IC$ is an algebraic subset defined by homogeneous equations expressed only using the ring structure of $\CA^*$. 
For rational sub-Hodge structures $\CA_1^l \seq \CA^l, \CA_2^k \seq \CA^k$, we will consider the following algebraic subsets: 

(1) $Z=\{\alpha \in \CA^k_\IC |\, \alpha^l=0 \} $;

(2) $Z=\{\alpha \in \CA^k_\IC | \cup\alpha: {\CA}^l_{1,\IC} \to \CA^{k+l}_\IC \,\text{vanishes} \} $; 

(3) $Z=\{\alpha \in \CA^k_{2,\IC} | \cup\alpha: {\CA}^l_{1,\IC} \to \CA^{k+l}_\IC \,\text{is not injective} \} $. 

\begin{lemma} [Deligne] \label{Deligne's lemma}
Let $Z_1\seq Z$ be an irreducible component. Assume the $\IC$-vector space $\la Z_1 \ra$ generated by the set $Z_1$ is defined over $\IQ$, that is, $\la Z_1 \ra= \CB^k \otimes \IC $ for some $\IQ$-vector space $\CB^k \seq \CA^k$. Then $\CB^k \seq \CA^k$ is a rational sub-Hodge structure. 
\end{lemma}

\begin{proof}
The Hodge decomposition of $\CA^k_\IC$ is the character decomposition of $\IC^*$-action on $\CA^k_\IC$ given by $z(\alpha)= z^p\bar{z}^q\alpha$ for $z\in\IC^*$ and $\alpha\in \CA^{p,q}$. 

If $\CB^k_\IC$ is stable under the $\IC^*$-action, then for any $\CB^k_\IC \ni b=\sum_{p+q=k}a^{p,q}$ with $a^{p,q} \in \CA^{p,q}$, we have $\CB^k_\IC \ni z(b)=\sum_{p+q=k}z^p\bar{z}^qa^{p,q} $. So for any polynomial $g(z)\in\IC[z] $, we have $\CB^k_\IC \ni g(z)(b)=\sum_{p+q=k}g(z^p\bar{z}^q)a^{p,q}$. Now we may choose $z_0\in \IC$ so that $z_0^p\bar{z}_0^q$ are distinct for all $p,q\ge 0$ with $p+q=k$. Then for any $p_0,q_0\ge 0$ with $p_0+q_0=k$, we may choose a polynomial $g(z) \in \IC[z]$ so that $g(z_0^{p}\bar{z}_0^{q})$ all vanish except $(p,q)=(p_0,q_0)$. Hence $a^{p_0,q_0}=g(z_0^{p_0}\bar{z}_0^{q_0})^{-1} g(z_0)(b) \in \CB^k_\IC$. So $\CB^k_\IC \seq \CA^k_\IC$ is a sub-Hodge structure. 

Now it suffices to show that $\CB^k_\IC=\la Z_1 \ra$ is stable under the $\IC^*$-action. First note that the $\IC^*$-action is compatible with the cup-product, that is, $z(\alpha\cup\beta)=z(\alpha)\cup z(\beta)$. Hence the algebraic subset $Z \seq \CA^k_\IC$ is stable under the $\IC^*$-action. Then its irreducible components are also stable under the $\IC^*$-action. So is $\la Z_1 \ra$. 
\end{proof}

We need another lemma to distinguish sub-Hodge structures in the $n=2$ case. 

\begin{lemma} \label{Deligne's lemma after}
Let V be a rational Hodge structure of weight $k$ with endomorphism  $\psi$ which is a morphism of Hodge structure. If $W\seq V$ is an invariant subspace of $\psi$ and $\psi|_W$ is diagonalizable, then $W\seq V$ is a sub-Hodge structure. 
\end{lemma}

\begin{proof}
It suffice to show $W_\IC=\oplus_{p+q=k} (V^{p,q}\cap W_\IC)$. Since $\psi|_W$ is diagonalizable, $W$ is generated by the eigenvectors of $\psi$. Note that $V^{p,q}$ are all eigenspaces of $\psi$. Hence each eigenvector of $\psi|_W$ is contained in some $V^{p,q}\cap W_\IC$. 
\end{proof}

\subsection{Cohomology ring structure of a blowup}

Let $X$ be a compact Kähler manifold, with $j_Z: Z\hookrightarrow X$ a submanifold of complex codimension $r$, and let $\tau: \widetilde{X} \to X$ be the blowup of $X$ along $Z$, with an exceptional divisor $j: E \hookrightarrow \widetilde{X}$. The restriction $\tau_E: E\to Z$ is then a $\mathbb{CP}^{r-1}$-bundle. Let $h=c_1(\CO_E(1))$.  We have a commutative diagram: 
\begin{equation*}
\xymatrix@R=3ex
{
E\ar[rr]^{j}
 \ar[dd]_{\tau_E}
&&
\widetilde{X} \ar[dd]^{\tau}\\
&&\\
Z\ar[rr]^{j_Z}&&
X. 
}    
\end{equation*}
There is a natural isomorphism $j_*: H^{2n-2}(E, \IZ) \cong H^{2n}(\widetilde{X},\IZ ) \cong \IZ $. Let $[E] \in H^2(\widetilde{X}, \IZ)$ be the cohomology class of $E$ and let $h=c_1(\CO_E(1))$. Then $j^*[E] = -h$ and $j_*1=[E]$. By the adjunction formula, $[E]^k = j_*1\cdot [E]^{k-1} = (-1)^{k-1}j_*(h^{k-1}) $. So for any $\alpha\in H^{2n-2k}(X,\IZ)$, we have 
$$[E]^{k}\cdot \tau^*\alpha= (-1)^{k-1}j_*(h^{k-1} \cdot j^*\tau^*\alpha) = (-1)^{k-1}j_*(h^{k-1} \cdot \tau_E^*(\alpha|_Z)). $$ 
If $k <r$, then $\alpha|_Z=0 $; if $k = r$, we have
$[E]^{r}\cdot \tau^*\alpha  = (-1)^{r-1}\alpha|_Z$. We have the equality: $h^r+h^{r-1}\tau^*_Ec_1+\cdots+\tau^*_Ec_r=0$, where $c_i=c_i(N_{Z/X})$ are the the Chern classes of the normal bundle of $Z$ in $X$. Hence for $k>r$, we see that $h^{k-1}$ is a linear combination of $1, h, \cdots, h^{r-1}$ whose coefficients are polynomials of $\{\tau^*_Ec_i\}_{i=0}^r$. We denote the coefficient of $h^{r-1}$ in $h^{k-1}$ by $\tau^*_Es_{k-r}$. Then one may compute $s_i$ by $\sum_{i=0}^{k-r}s_ic_{k-r-i}=0$ inductively. So $s_i$ is the $i$-th Segre class of $N_{Z/X}$. Now we have
$$[E]^{k}\cdot \tau^*\alpha= (-1)^{k-1}j_*(h^{k-1} \cdot \tau_E^*(\alpha|_Z))=(-1)^{k-1}j_*(h^{r-1}\cdot\tau^*_E(s_{k-r}\cdot\alpha|_Z)). $$
Hence $[E]^{k}\cdot \tau^*\alpha= (-1)^{k-1}s_{k-r}\cdot\alpha|_Z $ for $k\ge r$.


\subsection{Hodge index theorem}

Here we state a version of the Hodge index theorem that will be used repeatedly. For the proof, we refer to \cite[Theorem 6.32]{Voisin02}. 

\begin{lemma}\label{Hodge index theorem}
Let $X$ be a compact Kähler manifold of dimension $d$ with Kähler class $c\in H^2(X,\IR)$. Then $H^2(X,\IR)$ has Hodge-Lefschetz decomposition: 
$$H^2(X,\IR)=(H^{2,0}\oplus H^{0,2})(X,\IR) \oplus H^{1,1}(X,\IR)_{\rm prim} \oplus c \cdot H^0(X,\IR). $$
The quadratic form $q_c: \alpha \mapsto c^{d-2}\alpha^2$ is positive definite on $(H^{2,0}\oplus H^{0,2})(X,\IR) \oplus c \cdot H^0(X,\IR)$, and negative definite on $H^{1,1}(X,\IR)_{\rm prim}$. 
\end{lemma}

\section{Simply connected counterexamples of Kodaira problem}

\subsection{Voisin's examples}

Let $T$ be the complex torus of dimension $n\ge2$ constructed in Section 2.1, and let $K$ be the associated Kummer manifold, that is, $K=\widetilde{T}/\pm 1$ where $\widetilde{T}$ is the blowup of $T$ at all the $2$-torsion points. We denote the set of exceptional divisors on $K$ by $\{\Delta_i\}_{i=1}^{2^{2n}}$. 

The endomorphism $\phi_T$ of $T$ induces a rational map $\phi_K$ from $K$ to itself. We use $K_d, K_\phi \subseteq K\ti K$ to denote the graph of ${\rm id}_K, \phi_K$  respectively. Let us first blow up $K\ti K$ along $K_d$, then along the proper transform of $K_\phi$. We finally get the manifold $X_2$ (the same notation as \cite{Voisin04}),  with exceptional divisors $\D_d, \D_\phi$. 

\begin{thm}\label{main theorem}
Let $X_2'$ be a compact Kähler manifold. If there is a ring isomorphism 
$$\gamma: H^*(X_2',\IQ) \cong H^*(X_2,\IQ), $$
then $X_2'$ is not a projetive manifold. 
\end{thm}

Let $\tau: X_2 \to K\ti K $ be the natural morphism of blowup, and let ${\rm pr_i}: K\ti K \to K$ be the $i$-th projection map $(i=1,2)$. Recall the second cohomology group of the Kummer manifold: 
$$H^2(K,\IQ) = \wedge^2H^1(T,\IQ) \oplus \la \D_i \ra_{i=1}^{2^{2n}}. $$
\begin{nota}\rm
For $A,B\in \wedge^2H^1(T,\IQ)$ and $\Delta_i$, we simply write
$$A\ti K := \tau^*{\rm pr}_1^*A, \,\,\, K\ti B := \tau^*{\rm pr}_2^*B, $$
$$\D_i \ti K  := \tau^*{\rm pr}_1^*\Delta_i, \,\,\, K \ti \D_i   := \tau^*{\rm pr}_2^*\Delta_i.$$
\end{nota}
\noindent Then we have a decomposition: 
\begin{eqnarray*}
H^2(X_2, \IQ )
    &=& \la A \ti K \ra \oplus \la K \ti B \ra  \\
    & & \oplus \la \D_i \ti K \ra \oplus \la K \ti \D_i \ra    \oplus \la \Delta_d \ra \oplus \la \Delta_\phi \ra. 
\end{eqnarray*}
Let $\delta_d=\ga^{-1}(\Delta_d), \delta_\phi=\ga^{-1}(\Delta_\phi)$. We also have: 
\begin{eqnarray*}
H^2(X_2', \IQ )
    &=& \ga^{-1}\la A \ti K \ra \oplus \ga^{-1}\la K \ti B \ra  \\
    & & \oplus \ga^{-1}\la \D_i \ti K \ra \oplus \ga^{-1}\la K \ti \D_i \ra    \oplus \la \delta_d \ra \oplus \la \delta_\phi \ra. 
\end{eqnarray*}
\begin{lemma} \label{sub-Hodge structure lemma}
All the direct summands above of $H^2(X_2', \IQ)$ are sub-Hodge structures. 
\end{lemma}

The proof of Lemma \ref{sub-Hodge structure lemma} will follow from the cohomology ring structure of $H^*(X_2', \IQ)$ and Deligne's lemma. We first compute the cohomology ring structure of $H^*(X_2, \IQ)$, and prove a lemma which is essential in the proof of Theorem \ref{main theorem}.

\begin{lemma} \label{calculation lemma}
Let $C = \sum_i a_i (\D_i \ti K) + \sum_i b_i (K\ti \D_i) + u\D_d +v\D_\phi$. Then for any $\alpha=A\times K$ with $A\in \wedge^2H^1(T,\IQ)$, we have $C^{2n-1}\alpha=0$ for all $n\ge2$, $C^{2n-2}\alpha^2=0$ for all $n\ge3$. For $n=2$ we have 
$$C^2\alpha^2 = -\Big(2\sum_ib_i^2+u^2+v^2\Big)A^2. $$
\end{lemma}

\begin{proof}
Since the codimensions of $K_d$ and $K_\phi$ in $K\times K$ are both $n$, by the cohomology ring structure of a blowup, we have $\Delta_d^k\Psi=0, \Delta_\phi^k\Psi=0$ for all $k < n$, and 
$$\Delta_d^n(\eta \times \zeta) = (-1)^{n-1}\eta\cdot\zeta,\quad \Delta_\phi^n(\eta \times \zeta) = (-1)^{n-1}\eta\cdot\phi^*\zeta,$$ where $\Psi\in H^{4n-2k}(X_2, \IQ)$, $\eta\in H^{l}(K,\IQ), \zeta\in H^{2n-l}(K,\IQ)$ for some integer $0\le l \le 2n$.  In particular, $\Delta_d^k\Delta_\phi^l\Psi=0$ if $k<n$ or $l<n$ for any $\Psi \in H^{4n-2k-2l}(X_2,\IQ)$. 

Next we need to compute $\Delta_d^k\Psi, \Delta_\phi^k\Psi$ for $k> n$ and $\Psi\in H^{4n-2k}(K\times K, \IQ)$. By the cohomology ring structure of a blowup, we have 
$$\Delta_d^k\Psi=(-1)^{k-1}s_{k-n}^d\cdot\Psi|_{K_d},\quad \Delta_\phi^k\Psi=(-1)^ks_{k-n}^\phi\cdot\Psi|_{K_\phi},$$
where $s_{k-n}^d=s_{k-n}(N_{K_d/K\times K}), s_{k-n}^\phi=s_{k-n}(N_{K_\phi/K\times K})$ are the Segre classes of normal bundles of the graph $K_d,K_\phi$ respectively. 

Note that $K_d\cong K$,  and $K_\phi$ is just the blowup of $K$ along the undefined points $\{x_j\}$ of the rational map $\phi_K$. We use $\{E_j\}$ to denote the exceptional divisors. Hence the cohomology ring of $K_\phi$ is generated by $\wedge^2H^1(T,\IQ) \oplus \la \Delta_i \ra \oplus \la E_j \ra$.

\begin{claim}
For any $l>0$, $s_l^d \in \la \Delta_i^l \ra,s_l^\phi \in \la \Delta_i^l \ra\oplus \la E_j^l \ra$.
\end{claim}

Assuming the claim, we have $A\cdot s_l^d=0, A\cdot s_l^\phi=0$ for all $l>0$ for any $A\in \wedge^2H^1(T,\IQ)$. We shall finish the proof of the lemma. 

Write $C = \Lambda + \Delta$ where $\Lambda = \sum_i a_i (\D_i \ti K) + \sum_i b_i (K\ti \D_i), \Delta=u\D_d +v\D_\phi$. So for any $\alpha=A\times K$ with $A \in \wedge^2H^1(T,\IQ)$, we have: 
\begin{eqnarray*}
C^{2n-1}\alpha 
 &=& \Lambda^{2n-1}\alpha + \sum_{k=n}^{2n-1}\binom{2n-1}{k}\Delta^k\Lambda^{2n-1-k}\alpha\\
 &=& \Lambda^{2n-1}\alpha + \sum_{k=n}^{2n-1}\binom{2n-1}{k}(u^k\Delta_d^k + v^k\Delta_\phi^k)\Lambda^{2n-1-k}\alpha. 
\end{eqnarray*}
One may compute each term explicitly: 
\begin{eqnarray*}
\Lambda^{2n-1}\alpha
 &=& \sum_i b_i\, A\times \Delta_i^{2n-1}, \\
\Delta_d^k\Lambda^{2n-1-k}\alpha
 &=& \sum_ib_i\, (-1)^{k-1}\,s_{k-n}^d A\Delta_i^{2n-1-k}, \\
\Delta_\phi^k\Lambda^{2n-1-k}\alpha
 &=& \sum_ib_i\, (-1)^{k-1}s_{k-n}^\phi A\phi^*\Delta_i^{2n-1-k}, 
\end{eqnarray*}
hence $C^{2n-1}\alpha=0$ for all $n\ge2$. 
Similarly 
\begin{eqnarray*}
C^{2n-2}\alpha^2 =\Lambda^{2n-2}\alpha^2 + \sum_{k=n}^{2n-2}\binom{2n-2}{k}(u^k\Delta_d^k + v^k\Delta_\phi^k)\Lambda^{2n-2-k}\alpha^2,  
\end{eqnarray*}
hence $C^{2n-2}\alpha^2=0$ for all $n\ge 3$. For $n=2$, the exceptional divisors $\Delta_i$ are all $(-2)$-curves on $K$. We conclude that 
$C^2\alpha^2=-\Big(2\sum_ib_i^2+u^2+v^2\Big) A^2.$
\end{proof}
\renewcommand{\proofname}{Proof of Claim}
\begin{proof}
We have short exact sequences: 
$$0\to T_{K_d} \to T_{K\times K}|_{K_d} \to N_{K_d/K\times K} \to 0, $$
$$0\to T_{K_\phi} \to T_{K\times K}|_{K_\phi} \to N_{K_\phi/K\times K} \to 0. $$
Since $T_{K\times K}|_{K_d} \cong T_{K_d}\oplus T_{K_d}$, $T_{K\times K}|_{K_\phi} \cong {\rm pr}_1^*T_K|_{K_\phi}\oplus {\rm pr}_2^*T_K|_{K_\phi}$, we have 
$$s(N_{K_d/K\times K})=s(T_{K_d}),\quad s(N_{K_\phi/K\times K})= s({\rm pr}_1^*T_K|_{K_\phi})s({\rm pr}_2^*T_K|_{K_\phi})c(T_{K_\phi}).$$ 


We first compute the Segre class of $T_K$. Recall that $\widetilde{T}$ is the blowup of the $n$-dimensional complex torus $T$ at all the $2$-torsion points. Let $\pi:\widetilde{T}\to K $ be the quotient map, then the restriction $\pi:\widetilde{T}^*\to K^*$ is étale where $\widetilde{T}^*=\widetilde{T}\setminus \pi^{-1}(\cup_i\Delta_i), K^*=K\setminus(\cup_i \Delta_i)$. So $\pi^*s(T_K|_{K^*})=s(T_{\widetilde{T}^*})=1$. Hence $s(T_K|_{K^*})=1$ and $s_l(T_K)$ is supported in $\cup_i\Delta_i$ for all $l>0$. We conclude that $s_l(T_K)\in \la \Delta_i^l \ra$ for all $l>0$. 

Note that ${\rm pr}_1:K_\phi\to K$ is a blowup along the points $\{x_j\} \seq K^*$ and ${\rm pr}_2: K_\phi \to K$ is a ramified covering. We denote $K_\phi\setminus (\cup_i\Delta_i\cup\cup_jE_j) $ by $K_\phi^*$. Then the restrictions ${\rm pr}_1: K_\phi^*\to K^*\setminus \{x_j\}$ and ${\rm pr}_2: K_\phi^* \to K^*$ are both étale. We conclude as above that $s_l({\rm pr}_1^*T_K|_{K_\phi}), s_l({\rm pr}_2^*T_K|_{K_\phi}), c_l(T_{K_\phi})$ are all supported in $\cup_i\Delta_i\cup\cup_jE_j$. Hence they belong to $\la \Delta_i^l \ra\oplus \la E_j^l \ra$ for $l>0$. 
\end{proof}

Now we shall prove Lemma \ref{sub-Hodge structure lemma}. To simplify notations, we may write
$$\CA^2_1 = \ga^{-1}\la A \ti K\ra, \,\,\,
\CB^2_1 = \ga^{-1}\la A \ti K\ra \oplus \ga^{-1}\la \Delta_i \ti K\ra, $$
$$\CA^2_2 = \ga^{-1}\la K \ti B\ra, \,\,\,
\CB^2_2 = \ga^{-1}\la K \ti B\ra \oplus \ga^{-1}\la K\ti\Delta_i \ra, $$
$$\CA^2   = \CA^2_1 \oplus \CA^2_2, \,\,\,
\CB^2   = \CB^2_1 \oplus \CB^2_2, $$
$$\CP = \ga^{-1}\la \D_i \ti K\ra \oplus \ga^{-1}\la K \ti \D_i\ra \oplus \la \delta_d \ra \oplus \la \delta_\phi \ra. $$



\renewcommand{\proofname}{Proof of Lemma 6}
\begin{proof}
For $n\ge 3$, consider the algebraic subsets: 
\begin{eqnarray*}
Z &:=& \{\eta\in H^2(X_2',\IC): \eta^2 = 0 \}; \\
Y &:=& \{\zeta \in \CP_\IC: \cup \zeta: \CA^2_\IC \to H^4(X_2',\IC)\, \text{is not injective} \}. 
\end{eqnarray*}
They have irreducible decompositions: 
\begin{eqnarray*}
Z &=& Z_1\cup Z_2,  \\
Y &=& \ga^{-1}\la \Delta_i \ti K \ra_\IC \cup \ga^{-1}\la K \ti \Delta_i\ra_\IC    \cup \la \de_d \ra_\IC \cup \la \de_\phi \ra_\IC, 
\end{eqnarray*}
where $Z_1\subseteq \CA^2_{1,\IC}, Z_2\subseteq \CA^2_{2,\IC}$ are irreducible quadratic hypersurfaces, and hence generate the latter $\IC$-vector spaces. Deligne's lemma then implies that $\CA^2_1, \CA^2_2\subseteq H^2(X_2',\IQ)$ are sub-Hodge structures.  Let $\CA^* \subseteq H^*(X_2',\IQ)$ be the subalgebra generated by $\CA^2$. By Poincare duality, we have $\CA^{4n-2\perp}_\IC = \CP_\IC$. So by Deligne's lemma, $\CP \subseteq H^2(X_2',\IQ)$ is a sub-Hodge structure. Applying Deligne's lemma again, each irreducible component of $Y$ generates a sub-Hodge structure. Hence $\ga^{-1}\la \Delta_i \ti K\ra, \ga^{-1}\la K \ti \Delta_i\ra,$ $\la \de_d \ra,  \la \de_\phi \ra \subseteq H^2(X_2',\IQ)$ are all sub-Hodge structures. 

For $n=2$, let $\CB^*\seq H^*(X_2',\IQ)$ be the subalgebra generated by $\CB^2$. Consider the algebraic subsets: 
\begin{eqnarray*}
Z &:=& \{\eta\in H^2(X_2',\IC): \eta^2 = 0 \}; \\
Y' &:=& \{\zeta \in \la\de_d \ra_\IC \oplus \la \de_\phi \ra_\IC: \cup \zeta: \CB^2_\IC \to H^4(X_2',\IC)\, \text{is not injective} \}. 
\end{eqnarray*}
We have irreducible decompositions $Z=Z_1\cup Z_2$, $Y'=\la \de_d \ra_\IC \cup \la \de_\phi \ra_\IC$, where $Z_1\seq \CB^2_{1,\IC}, Z_2\seq \CB^2_{2,\IC}$ are irreducible quadratic hypersurfaces. Note also $\la\de_d \ra \oplus \la \de_\phi \ra = \CB^{6\perp}$.  Hence $\CB^2_1, \CB^2_2, \la \de_d \ra, \la \de_\phi \ra \subseteq H^2(X_2',\IQ)$ are sub-Hodge structures. 

Now ${\rm Ker}(\cup\delta_d)$ induces an isomorphism of Hodge structures $\psi_d: \CB^2_1\to\CB^2_2$, with $\psi_d(\CA^2_1)=\CA^2_2, \psi_d(\ga^{-1}\la \Delta_i\times K \ra)=\ga^{-1}\la K\times\Delta_i \ra$, and ${\rm Ker}(\cup\delta_\phi)$ induces a morphism of Hodge structures $\psi_\phi: \CB^2_1\to\CB^2_2$, with $\psi_\phi(\CA^2_1)=\CA^2_2, \psi_\phi(\ga^{-1}\la \Delta_i\times K \ra)=\ga^{-1}\la K\times\Delta_i \ra$. We conclude that $\CB^2_1$ admits an endomorphism $\psi$ of Hodge structure which is compatible with $\phi^*_K$ on $H^2(K,\IQ)$. Hence $\CA^2_1, \ga^{-1}\la \D_i \ti K \ra$ are invariant subspaces of $\psi$, and $\psi|_{\CA^2_1}$ is diagonalizable. So $\CA^2_1$ is a sub-Hodge structure of $\CB^2_1$ by Lemma \ref{Deligne's lemma after}. Finally, since $\{\zeta\in \CB^2_{1,\IC}: \zeta \cup \CA^2_{1,\IC} =0 \}=\ga^{-1}\la \D_i \ti K \ra_\IC$, we see that $\ga^{-1}\la \D_i \ti K \ra \seq \CB^2_1$ is a sub-Hodge structure. Then $\CA^2_2, \ga^{-1}\la K \ti \D_i \ra\seq \CB^2_2$ are sub-Hodge structures. 
\end{proof}
\renewcommand{\proofname}{Proof}


Since $\de_d, \de_\phi\in H^2(X_2',\IQ) $ are Hodge classes, the maps $\cup \de_d, \cup \de_\phi: \CA^2 \to H^4(X_2',\IQ)$ are morphisms of Hodge structures, and their kernels are sub-Hodge structures of $\CA^2$. So we conclude as the previous paragraph that $\CA^2_1$ and $\CA^2_2$ are isomorphic Hodge structures, and $\CA^2_1$ admits an endomorphism of Hodge structure which is compatible with $\phi^*_T = \wedge^2\,^{t}\phi$ on $\wedge^2H^1(T,\IQ)$. 

Applying Lemma \ref{no Hodge class lemma} and the following lemma, we see that any Hodge class of $H^2(X_2',\IQ)$ is contained in $\CP$.

\begin{lemma} \label{Hodge str nontrivial lemma}
The Hodge structure $\CA^2_1$ has non-trivial $(2,0)$ part. 
\end{lemma}
\begin{proof}
Assume that the $(2,0)$ part of the Hodge structure $\CA^2_1$ vanishes. We will prove that $X_2'$ is not Kähler. 

Let $c\in H^2(X_2',\IR)$ be a Kähler class. Since the Hodge structure $\CA_1^2$ has no $(2,0)$ part, we have $\CA_{1,\IR}^2 \seq H^{1,1}(X_2', \IR) = \la c \ra \oplus H^{1,1}(X_2',\IR)_{\rm prim}$. By Hodge index theorem, the quadratic form $q_c$ is positive definite on $\la c \ra$ and negative definite on $H^{1,1}(X_2',\IR)_{\rm prim}$. So $\CA^2_{1,\IR}$ contains at most $1$-dimensional isotropic subspace with respect to $q_c$. 

But there is a $2$-dimensional subspace $V\seq \CA^2_{1,\IR} $ such that $\alpha^2=0$ for all $\alpha \in V$. For example, let $V= \gamma^{-1} \la (e_1\wedge e_2)\times K,(e_1\wedge e_3) \times K \ra$ where $\{e_i\}_{i=1}^{2n}$ is a basis of $H^1(T,\IR)$. So $V$ is isotropic with respect to $q_c$. We get a contradiction. 
\end{proof}

Now let us prove the main theorem. 

\renewcommand{\proofname}{Proof of Theorem 2}
\begin{proof}
Assume that $X_2'$ is projective. Then there exists an ample class $c\in H^2(X_2',\IQ)$. Lemma \ref{no Hodge class lemma} and Lemma \ref{Hodge str nontrivial lemma} imply that $c\in \CP$. Hence we may write $c = \sum_i a_i \ga^{-1}(\Delta_i \ti K) + \sum_i b_i \ga^{-1}(K\ti \Delta_i) + u\de_d +v\de_\phi$. Since $\CA_1^2$ has non-trivial $(2,0)$ part, we see that the quadratic form $q_c$ does not vanish on $\CA^2_1$ by Hodge index theorem. 

If $n\ge 3$, then we have $c^{2n-2}\CA_1^2=0$ by Lemma \ref{calculation lemma}. We get a contradiction. 

If $n=2$, Lemma \ref{calculation lemma} implies that for any $\ga^{-1}(A \times K) \in \CA_1^2 $, 
$$c^2\ga^{-1}(A\times K)^2 =-\Big(2\sum_ib_i^2 +u^2+v^2 \Big)A^2, $$
with $2\sum_ib_i^2 +u^2+v^2 >0$. Hence $q_c$ has signature $(3,3)$ on $\CA_{1,\IR}^2$. 

Since $c^3\CA_1^2=0$, we see that $\CA_1^2 \subseteq H^2(X_2',\mathbb{Q})_{\rm prim}$. We have Hodge decomposition
$$\CA_{1, \IR}^2=(\CA_1^{2,0}\oplus \CA_1^{0,2})_\IR \oplus \CA_{1,\IR}^{1,1}. $$
By the Hodge index theorem, $q_c$ is positive definite on $(\CA_1^{2,0}\oplus \CA_1^{0,2})_\IR$ and negative definite on $\CA_{1,\IR}^{1,1}$. Hence $q_c$ has signature $(2a,6-2a)$ on $\CA_{1,\IR}^2$ where $a=\dim_{\IC} \CA_1^{2,0}$. We get a contradiction for $2a\ne 3$.  
\end{proof}
\renewcommand{\proofname}{Proof}


\subsection{Examples given by product}
Let $X_3=X_2\times M$ where $M$ is a compact Kähler manifold with second Betti number $b_2(M)=1$. We will prove that $X_3$ does not have the rational homotopy type of a complex projective manifold.

\begin{thm} \label{main theorem 3}
Let  $X_3'$ be a compact Kähler manifold. If there is a ring isomorphism 
$$\gamma: H^*(X_3',\IQ) \cong H^*(X_3,\IQ), $$
then $X_3'$ is not a projetive manifold. 
\end{thm}

By Künneth formula, we have $H^2(X_3,\IQ)=H^2(X_2,\IQ) \oplus \la h \ra$, where $h$ is a generater of $H^2(M,\IQ)$. We see that $h^m\ne 0$. With the same notations in Section 3.1, we have: 
\begin{eqnarray*}
H^2(X_3, \IQ )
    &=& \la A \ti K  \ra 
    \oplus \la K \ti B  \ra
    \oplus \la h \ra  \\
    & & \oplus \la \D_i \ti K  \ra 
    \oplus \la K \ti \D_i  \ra    
    \oplus \la \Delta_d  \ra 
    \oplus \la \Delta_\phi  \ra,\\
H^2(X_3', \IQ )
    &=& \ga^{-1}\la A \ti K\ra 
    \oplus \ga^{-1}\la K \ti B\ra  
    \oplus \ga^{-1}\la h \ra \\
    & & \oplus \ga^{-1}\la \D_i \ti K\ra 
    \oplus \ga^{-1}\la K \ti \D_i\ra    
    \oplus \la \delta_d \ra 
    \oplus \la \delta_\phi \ra. 
\end{eqnarray*}

\begin{lemma} 
All the direct summands above of $H^2(X_3', \IQ)$ are sub-Hodge structures. 
\end{lemma}

\begin{proof}
We first note that $h\cup\alpha \ne 0$ for any non-zero $\alpha\in \ga^{-1}H^*(X_2, \IQ)$. Hence $\ga^{-1}\la h \ra_{\IC}$ is an irreducible component of the algebraic subset $\{\eta\in H^2(X_3',\IC): \eta^{m+1} = 0 \}$. So $\ga^{-1}\la h \ra \seq H^2(X_3', \IQ)$ is a sub-Hodge structure by Deligne's lemma. 


With the same argument in Lemma \ref{sub-Hodge structure lemma}, one shows that the other six direct summands of $H^2(X_3',\IQ)$ are sub-Hodge structures. 
\end{proof}


We conclude that $\CA^2_1=\ga^{-1}\la A\times K\ra$ contains no Hodge class as Section 3.1. So the proof of Theorem \ref{main theorem 3} follows from the following analogous of Lemma \ref{calculation lemma}.

\begin{lemma}
Let $C = h+ \sum_i a_i (\D_i \ti K) + \sum_i b_i (K\ti \D_i) + u\D_d +v\D_\phi$. Then for any $\alpha=A\times K$ with $A\in \wedge^2H^1(T,\IQ)$, we have $C^{2n+m-1}\alpha=0$ for all $n\ge2$, $C^{2n+m-2}\alpha^2=0$ for all $n\ge3$. For $n=2$ we have 
$$C^{2+m}\alpha^2 = -q(h^m)\Big(2\sum_ib_i^2+u^2+v^2\Big)A^2, $$
where $q:H^{2m}(M,\IQ)\cong \IQ$ is an isomorphism. 
\end{lemma}

\subsection{Oguiso's examples of dimension four}

Here we give a slightly more general construction of Oguiso's example. 

Let $(S,\phi)$ be a pair where $S$ is a K3 surface with elliptic $NS(S)$ (the natural pairing of $H^2(S,\IQ)$ is negative definite on $NS(S)_\IQ$) and $\phi$ is an automophism of $S$ of infinite order. Here we need to assume that the Picard number $\rho(S)\le 16$ for some technical reason. We use $S_d, S_\phi\seq S\times S$ to denote the graph of $id_S, \phi$ respectively. 

We first blow up $S\times S$ along $S_d\cap S_\phi$, then along the proper transform of $S_d, S_\phi$. We get the manifold $X_4$. This is Oguiso's example when $(S,\phi)$ is a McMullen's pair (we refer to \cite{Oguiso08} for the definition of McMullen's pair). We will prove that $X_4$ does not have the rational homotopy type of a complex projective manifold.

\begin{thm} \label{main theorem 4}
Let  $X_4'$ be a compact Kähler manifold. If there is a ring isomorphism 
$$\gamma: H^*(X_4',\IQ) \cong H^*(X_4,\IQ), $$
then $X_4'$ is not a projetive manifold. 
\end{thm}

We first state a basic result in \cite[Section 3]{Oguiso08}. Let $N:=NS(S)_\IQ$, $T:=N^\perp \seq H^2(S,\IQ)$. They are invariant subspaces of $\phi^*$. We use $f$ to denote the characteristic polynomial of $\phi^*|_T$. 

\begin{prop}[Oguiso] \label{Oguiso}
The polynomial $f$ is irreducible and ${\rm deg}f$ is even. In particular, $\phi^*|_T$ is diagonalizable and $\dim_\IQ T$ is even.  
\end{prop}

Let $ \tau: X_4\to S\times S$ be the natural map and ${\rm pr}_i: S\times S \to S$ be the natural projection for $i=1,2$. We use $\Delta_d, \Delta_\phi$ and $\{E_i\}$ to denote the irreducible exceptional divisors of $X_4$ with centers $S_d, S_\phi$ and $S_d\cap S_\phi$ respectively. Hence we have decompositions:
\begin{eqnarray*}
H^2(S,\IQ) &=& T\oplus N, \\
H^2(X_4,\IQ) &=& \ga(T_1) \oplus \ga(T_2) \oplus \ga(N_1) \oplus \ga(N_2)\oplus \la\Delta_d\ra \oplus \la\Delta_\phi\ra \oplus \la E_i\ra, \\
H^2(X_4',\IQ) &=& T_1 \oplus T_2 \oplus N_1 \oplus N_2\oplus \la\delta_d\ra \oplus \la\delta_\phi\ra \oplus \la e_i\ra, 
\end{eqnarray*} 
where $T_i=\ga^{-1}\tau^*{\rm pr}_i^*T, N_i=\ga^{-1}\tau^*{\rm pr}_i^*N, \delta_d=\ga^{-1}(\Delta_d), \delta_\phi=\ga^{-1}(\Delta_\phi), e_i=\ga^{-1}(E_i)$. 

\begin{lemma} \label{sub-Hodge structure lemma 3}
All the direct summands above of $H^2(X_4', \IQ)$ are sub-Hodge structures. 
\end{lemma}

The proof is similar to the $n=2$ case of Lemma \ref{sub-Hodge structure lemma}. 

\begin{proof}
Let $\CB^2_1=T_1\oplus N_1,  \CB^2_2= T_2\oplus N_2$, $\CB^2=\CB^2_1\oplus \CB^2_2$, and let $\CB^*\seq H^*(X_4', \IQ)$ be the subalgebra generated by $\CB^2$. 
The algebraic subset $Z=\{\eta\in H^2(X_4',\IC): \eta^2=0\}$ has irreducible decomposition $Z=Z_1\cup Z_2$, where $Z_1\seq \CB^2_{1,\IC}, Z_2\seq \CB^2_{2,\IC}$ are irreducible quadratic hypersurfaces. Hence $\CB^2_1, \CB^2_2 \seq H^2(X_4',\IQ)$ are sub-Hodge structures by Deligne's lemma. The algebraic subset $\{\zeta\in H^2(X_4',\IC): \zeta\cup (T_1\oplus N_1\oplus T_2\oplus N_2)=0\}$ is just $\la e_i\ra_\IC$. So $\la e_i\ra \seq H^2(X_4',\IQ)$ is a sub-Hodge structure. Denote $\CH=T_1 \oplus T_2 \oplus N_1 \oplus N_2\oplus \la\delta_d\ra \oplus \la\delta_\phi\ra$. Then we have $\{\zeta\in H^2(X_4',\IC): \zeta\cup \la e_i\ra =0\}=\CH_\IC$. Hence $\CH\seq H^2(X_4',\IQ)$ is a sub-Hodge structure. 

Now we show that the direct summands of $\CH$ are all sub-Hodge structures.  Note that $(\la\de_d \ra \oplus \la \de_\phi \ra)_\IC = \{\zeta\in \CH_\IC: \zeta\cup B^6=0\}$. Hence $\la\de_d \ra \oplus \la \de_\phi \ra \seq \CH$ is a sub-Hodge structure. The algebraic subset $\{\zeta \in (\la\de_d \ra \oplus \la \de_\phi \ra)_\IC: \cup \zeta: \CB^2_\IC \to H^4(X_4',\IC)\, \text{is not injective} \}$ is just $\la \de_d \ra_\IC \cup \la \de_\phi \ra_\IC$.  Hence $\la \de_d \ra, \la \de_\phi \ra \subseteq H^2(X_2',\IQ)$ are sub-Hodge structures. 

As the proof of Lemma \ref{sub-Hodge structure lemma}$, {\rm Ker}(\cup\delta_d)$ induces an isomorphism of Hodge structures $\psi_d: \CB^2_1\to\CB^2_2$, with $\psi_d(T_1)=T_2, \psi_d(N_1)=N_2$, and ${\rm Ker}(\cup\delta_\phi)$ induces an isomorphism of Hodge structures $\psi_\phi: \CB^2_1\to\CB^2_2$, with $\psi_\phi(T_1)=T_2, \psi_\phi(N_1)=N_2$. We conclude that $B^2_1$ admits an automorphism $\psi$ of Hodge structure which is compatible with $\phi^*$ on $H^2(S,\IQ)$. Hence $T_1, N_1$ are invariant subspaces of $\psi$, and $\psi|_{T_1}$ is diagonalizable. So $T_1\seq \CB^2_1$ is a sub-Hodge structure by Lemma \ref{Deligne's lemma after}. Since $\{\zeta\in \CB^2_1: \zeta \cup T_1 =0 \}=N_1$, we see that $N_1 \seq \CB^2_1$ is a sub-Hodge structure. Then $T_2, N_2\seq \CB^2_2$ are sub-Hodge structures. 
\end{proof}

We have shown in the previous paragraph that $T_1$ admits an automorphism $\psi$ of Hodge structure which is compatible with $\phi^*|_T$ on $T$.

\begin{lemma}
The Hodge structure $T_1$ has non-trivial $(2,0)$ part. 
\end{lemma}

\begin{proof}
Since $N=NS(S)_\IQ$ is elliptic, that is, the natural pairing is negative definite on $N$, we see that the natural pairing has signature $(3,t-3)$ on $T_\IR$, where $t=\dim_\IQ T\ge 6$ (the assumption $\rho(S)\le 16$ is only used here). Hence $T_\IR$ contains an isotropic subspace $W_\IR$ of dimension $3$. We see that $W_{1,\IR}=\ga^{-1}\tau^*{\rm pr_1^*}W_\IR \seq T_{1,\IR}$ is isotropic with respect to $q_c$ for any Kähler class $c\in H^2(X_4',\IR)$. 

Assume $T_1$ has trivial $(2,0)$ part. By Hodge index theorem, the quadratic form $q_c$ has signature $(1, t-1)$ or $(0,t)$ on $T_{1,\IR}$. Hence $T_{1,\IR}$ contains at most $1$-dimensional isotropic subspaces. We get a contradiction for $W_{1,\IR}\seq T_{1,\IR}$ is a $3$-dimensional isotropic subspace. Hence $T_1$ has non-trivial $(2,0)$ part. 
\end{proof}

\begin{lemma} \label{no Hodge class lemma 4}
The Hodge structure $T_1$ contains no Hodge class.
\end{lemma}

\begin{proof}
Assume $W=T\cap T^{1,1}$ is non-zero. Then $W$ is an invariant subspace of $\psi$. Since $T^{2,0}\ne 0$, we have $W\ne T$. Hence the characteristic polynomial of $\psi$ is reducible, which contradicts Proposition \ref{Oguiso}. 
\end{proof}

\begin{lemma} \label{computation lemma 4}
For any class $c\in \CP=N_1 \oplus N_2\oplus \la\delta_d\ra \oplus \la\delta_\phi\ra \oplus \la e_i\ra$, we have{\rm :}

{\rm (1)} The quadratic form $q_c$ on $T_1$ is proportional to the natural pairing on $T$, that is, there exists $\lambda \in \IQ$ such that $q_c(\ga^{-1}\tau^*{\rm pr}_1^* \alpha) = \lambda \alpha^2$ for any $\alpha \in T$. 

{\rm (2)} The subspace $T_1\seq H^2(X_4',\IQ)$ is primitive with respect to $c$, that is, $c^3T_1=0$.  
\end{lemma}
The proof is the same as Lemma \ref{calculation lemma}. 

\renewcommand{\proofname}{Proof of Theorem 4}
\begin{proof}
Assume that $X_4'$ is projective. By Lemma \ref{no Hodge class lemma 4}, there is an ample class $c\in \CP$. Hence the quadratic form $q_c$ has signature $(t-3,3)$ or $(3, t-3)$ on $T_{1,\IR}$ by Lemma \ref{computation lemma 4}, where $t=\dim_\IQ T$. Since $T_1\seq H^2(X_4',\IQ)_{\rm prim}$, the quadratic form $q_c$ has signature $(2a, t-2a)$ on $T_{1,\IR}$ by Hodge index theorem, where $a=\dim_\IC T_1^{2,0}$. We get a contradiction for $3, t-3$ are odd and $2a, t-2a$ are even. Hence $X_4'$ is not projective. 
\end{proof}

\subsection{Oguiso's examples of higher dimensions}

Let $X_5=X_4\times F_d$ where $F_d$ is a smooth hypersurface of degree $d+2$ in $\IC\IP^{d+1}$ for $d\ge 2$, and $F_1=\IC\IP^1$. This is the higher dimensional example that Oguiso constructed in \cite{Oguiso08}. We will show that $X_5$ does not have the rational homotopy type of a complex projective manifold.

\begin{thm} \label{main theorem 5}
Let  $X_5'$ be a compact Kähler manifold. If there is a ring isomorphism 
$$\gamma: H^*(X_5',\IQ) \cong H^*(X_5,\IQ), $$
then $X_5'$ is not a projetive manifold. 
\end{thm}

By Kunneth formula, we have $H^2(X_5, \IQ)= H^2(X_4, \IQ) \oplus H^2(F_d, \IQ)$. Hence we get a decomposition: 
$$H^2(X_5',\IQ) = T_1 \oplus T_2 \oplus N_1 \oplus N_2\oplus \la\delta_d\ra \oplus \la\delta_\phi\ra \oplus \la e_i\ra \oplus H,$$
where $H=\ga^{-1}H^2(F_d, \IQ)$. One may prove that all the direct summands above are sub-Hodge structures. 

Indeed, 
for $d=1$, we see that $H=\la h\ra$ where $h=\ga^{-1}c_1(\CO_{\IC\IP^1}(1))$, and $H_\IC$ is an irreducible component of $\{\eta\in H^2(X_5',\IC) :\eta^{2}=0 \}$; 
for $d=2$, $F_2$ is a K3 surface, hence one of the irreducible component of $\{\eta\in H^2(X_5',\IC): \eta^2=0\}$ is a quadratic hypersurface in $H_\IC$; 
for $d\ge3$, we have $\dim_\IQ H=1$ by Lefschetz hyperplane theorem, so $H_\IC$ is an irreducible component of $\{\eta\in H^2(X_5',\IC) :\eta^{d+1}=0 \}$. Hence $H\seq H^2(X_5',\IQ)$ is a sub-Hodge structure by Deligne's lemma. With the same argument in Lemma \ref{sub-Hodge structure lemma 3}, we see that $T_1, $$ T_2, $$ N_1, $$ N_2,$$ \la\delta_d\ra, $$ \la\delta_\phi\ra, $ $ \la e_i\ra \seq H^2(X_5', \IQ)$ are sub-Hodge structures. 

With the same argument in Section 3.3, one shows that the Hodge structure $T_1\oplus T_2$ contains no Hodge class.  Hence an analogous of Lemma \ref{computation lemma 4} implies that $X_5'$ is not projective.



\end{document}